\documentclass[a4paper, 10pt]{article}
\usepackage{graphicx}
\usepackage{alltt}
\usepackage{amsmath}
\usepackage[hidelinks, pdftex]{hyperref}
\usepackage{amsthm}
\usepackage{cite}
\usepackage{amssymb}

\newcommand{\doi}[1]{\href{https://doi.org/#1}{\texttt{https://doi.org/#1}}}

\newtheorem{theorem}{Theorem}[section]
\newtheorem{lemma}[theorem]{Lemma}

\theoremstyle{definition}
\newtheorem{definition}[theorem]{Definition}
\newtheorem{example}[theorem]{Example}
\newtheorem{proposition}[theorem]{Proposition}
\newtheorem{corollary}[theorem]{Corollary}

\theoremstyle{remark}
\newtheorem{remark}[theorem]{Remark}

\DeclareMathOperator{\Fix}{Fix}
\begin{document}

\begin{center}
\LARGE
\textbf{Fixed point results for  single and multi-valued  three-points contractions}\footnote{The third author is supported by Researchers Supporting Project number (RSP2024R4), King Saud University, Riyadh, Saudi Arabia. This work was partially supported by a grant from the Simons Foundation (PD-Ukraine-00010584, E. Petrov)}\\[6pt]
\small
\textbf {Mohamed Jleli\textsuperscript{a}, Evgeniy Petrov\textsuperscript{b,}\footnote{Corresponding author.}, Bessem Samet\textsuperscript{a}}\\[6pt]
\textsuperscript{a}Department of Mathematics, College of Science, King Saud University, Riyadh 11451, Saudi Arabia \\ jleli@ksu.edu.sa, bsamet@ksu.edu.sa\\[6pt]
\textsuperscript{b}Function Theory Department, Institute of Applied Mathematics and Mechanics of the NAS of Ukraine, Batiuka str. 19, Slovyansk, 84116, Ukraine \\ eugeniy.petrov@gmail.com\\[6pt]
Received: date\quad/\quad
Revised: date\quad/\quad
Publised online: data
\end{center}

\begin{abstract}
In this paper, we are concerned with the study of the existence of fixed points for  single and multi-valued  three-points contractions. Namely, we first introduce a new class of single-valued mappings defined on a metric space equipped with three metrics.  A fixed point theorem is established for such mappings. The obtained result recovers that established  recently by the second author [J. Fixed Point Theory Appl. 25 (2023) 74] for the class of single-valued mappings  contracting perimeters of triangles. We next extend our study by introducing the class of multivalued three points contractions.  A fixed point theorem, which is a multi-valued version of that obtained in the above reference, is established. Some examples showing the validity of our obtained results are provided. \vskip 2mm


\end{abstract}

\section{Introduction}\label{sec1}

Banach's contraction principle \cite{Banach} is one of the most celebrated fixed point theorems. This theorem  states that, if $F$ is a self-mapping  defined on a complete metric space $(M,d)$ and satisfies

\begin{equation}\label{e0}
d(Fu,Fv)\leq \lambda d(u,v),\quad (u,v)\in M\times M,	
\end{equation}
where $\lambda\in [0,1)$ is a constant, then $F$ possesses one and only one  fixed point. Moreover, for any $u_0\in M$, the Picard sequence $\{u_n\}\subset X$
defined by $u_{n+1}=Fu_n$, $n\geq 0$, converges to the unique fixed point of $F$.
The mapping $F$ satisfying inequality~(\ref{e0}) with $\lambda\in [0,1)$ is called a contraction. The literature includes  several interesting generalizations and extensions of the above result. We may distinguish at least two categories of generalizations and extensions. In the first one, the contractive nature of the mapping is weakened, see e.g. the series of papers \cite{BO,Ciric,KH,Kirk,PO,PRO,RA,Rus2} and the monograph \cite{AG}. In the second category, the topology of the underlying space is weakened, see e.g. \cite{BR,CZ,JS,Mustafa,OL}.

The study of fixed points for multi-valued mappings was first considered in the paper \cite{Nadler} by Nadler, where he proved the following interesting result.
\begin{theorem}[Nadler (1969)] \label{T1.1}
Let $(M,d)$ be a complete metric space and $F: M\to CB(M)$ be a given multi-valued mapping, where $CB(M)$ denotes the family of all nonempty bounded and closed subsets of $M$. Assume that
$$
H(Fu,Fv)\leq \lambda d(u,v),\quad (u,v)\in M\times M,	
$$	
where $\lambda\in (0,1)$ is a constant and $H$ is the Hausdorff-Pompeiu metric on $CB(M)$. Then  $F$ possesses at least one fixed point.
\end{theorem}
Theorem \ref{T1.1} was generalized and extended in various directions, see e.g. \cite{AS,Feng,IT,KA,MIZ,PE,Rus0}. More recent references can be found in \cite{PE2,TA}.

Recently, the second author \cite{Petrov} obtained a generalization of Banach's fixed point theorem by introducing the class of single-valued mappings  contracting perimeters of triangles (three-points contractions). Namely, he studied the existence of fixed points for the following class of mappings.

\begin{definition}\label{def1.2}
Let $(M,d)$ be a metric space with $|M|\geq 3$. A single-valued mapping $F: M\to M$ is said to be a mapping contracting perimeters of triangles on $M$, if there exists $\lambda \in [0,1)$ such that the inequality
$$
d(Fx,Fy)+d(Fy,Fz)+ d(Fz,Fx)\leq \lambda [d(x,y)+d(y,z)+d(z,x)],
$$
holds for all three pairwise distinct points $x,y,z \in M$.
\end{definition}

The following fixed point result was established in \cite{Petrov}.

\begin{theorem}\label{T1.3}
Let $(M,d)$ be a complete metric space with  $|M|\geq 3$. Let the
mapping $F: M \to M$  satisfies the following two conditions:
\begin{itemize}
\item[{\rm{(i)}}] For all $u\in M$, $F(Fu)\neq u$, provided $Fu\neq u$;
\item[{\rm{(ii)}}]	$F$ is a mapping contracting perimeters of triangles on $M$.
\end{itemize}
Then $\Fix(F)\neq\emptyset$ and $|\Fix(F)|\leq 2$, where  $\Fix(F)$ denotes the set of fixed points of $F$.
\end{theorem}
Other fixed point results related to mappings contracting perimeters of triangles, can be found in \cite{Bey,BI,PO2}. {Three-point analogs of the well-known Kannan and Chatterjea fixed-point theorems were considered in~\cite{PB24} and~\cite{PP24}, respectively.}

We fix below some notations and recall some basic definitions.
\\
Throughout this paper, by $\mathbb{R}^+$, we mean the interval $[0,\infty)$. We denote by $M$  and arbitrary nonempty set. By $|M|$, we mean the cardinal of $M$.   For a single-valued mapping $F: M\to M$, we denote by $\Fix(F)$ the set of its fixed points, that is,
$$
\Fix(F)=\{u\in M: Fu=u\}.
$$
We define the sequence of mappings $(F^n)$, where $F^n: M\to M$,  by
$$
F^0=I_M\,\, (\mbox{i.e. }F^0x=x\mbox{ for all } x\in M)
$$
and
$$
F^{n+1}=F\circ F^n,\quad n\geq 0.
$$
Similarly, for a given function $\varphi: \mathbb{R}^+\to \mathbb{R}^+$, we define  the sequence of functions $\{\varphi^n\}$, where $\varphi^n: \mathbb{R}^+\to \mathbb{R}^+$,  by
$$
\varphi^0=I_{\mathbb{R}^+}, \quad  \varphi^{n+1}=\varphi\circ \varphi^n,\quad n\geq 0.
$$

Let $(M,d)$ be a metric space. By $CB(M)$, we denote the  family of all nonempty bounded and closed subsets of $M$. The distance between  two subsets $A$ and $B$ of $M$ is denoted by $D(A,B)$, that is,
$$
D(A,B)=\inf\{d(a,b): a\in A, b\in B\}.
$$
The diameter of $A,B\in CB(M)$ is denoted by $\mathcal{D}(A,B)$, that is,
$$
\mathcal{D}(A,B)=\sup_{(a,b)\in A\times B}d(a,b).
$$
We denote by $H$ the Hausdorff-Pompeiu metric on  $CB(M)$ induced by $d$, that is,
$$
H(A,B)=\max\left\{\sup_{a\in A} D(a,B), \sup_{b\in B}D(b,A)\right\},\quad A,B\in CB(M).
$$
Recall that $(CB(M),H)$ is a metric space.  Moreover, if $(M,d)$ is complete, then $(CB(M),H)$ is also complete (see e.g. \cite{Rus2}).

Let $\mathcal{P}(M)$ be the family of all nonempty subsets of $M$. We say that $u\in M$ is a fixed point of a multi-valued mapping $F: M\to \mathcal{P}(M)$, if $u\in Fu$. We also denote by $\Fix(F)$ the set of fixed points of $F$, that is,
$$
\Fix(F)=\left\{u\in M: u\in Fu\right\}.
$$

The rest of the paper is organized as follows. In Section \ref{sec2}, we introduce a certain class  of single-valued three-points contractions on $M$, where $M$ is equipped with three metrics $d_i$, $i=1,2,3$. A fixed point theorem is established for this class of mappings. In particular, we recover Theorem \ref{T1.3}. In Section \ref{sec3}, we extend Theorem \ref{T1.3} from the single-valued case to the multi-valued case.

\section{Three-points single-valued contractions} \label{sec2}

Let us  denote by $\Phi$ the set of functions $\varphi: \mathbb{R}^+\to \mathbb{R}^+$ satisfying the following conditions:
\begin{itemize}
\item[(C$_1$)] 	$\varphi$ is nondecreasing;
\item[(C$_2$)] $\displaystyle \sum_{n\geq 1} \varphi^n(s)<\infty$ for every $s>0$.
\end{itemize}

\begin{remark}
(i) From (C$_2$), we deduce that, if  $\varphi\in \Phi$, then
\begin{equation}\label{pp1}
\lim_{n\to \infty}\varphi^n(s)=0,\quad s>0.
\end{equation}	\vspace{-0.5cm}
(ii)  If $\varphi\in \Phi$, then
\begin{equation}\label{pp2}
\varphi(s)<s,\quad s>0.
\end{equation}
Indeed, if there exists $s>0$ such that $s\leq \varphi(s)$, then by (C$_1$), we get
$$
s\leq \varphi^n(s),\quad n\geq 0.
$$
Using \eqref{pp1} and passing to the limit as $n\to \infty$ in the above inequality, we reach a contradiction with $s>0$, which proves \eqref{pp2}. \\
(iii) From \eqref{pp2}, we deduce immediately that
$$
\lim_{s\to 0^+} \varphi(s)=0.
$$
\end{remark}

Some examples of functions $\varphi\in \Phi$ are provided below.

\begin{example}\label{ex2.2}
A basic example of a function $\varphi\in \Phi$ is the function
$$
\varphi(t)=\lambda t,\quad t\in \mathbb{R}^+,
$$	
where $\lambda\in (0,1)$ is a constant.
\end{example}

\begin{example}\label{ex2.3}
Let $\varphi: \mathbb{R}^+\to \mathbb{R}^+$ be the function defined by
$$
\varphi(t)=\frac{1}{2} \ln(t+1),\quad t\in \mathbb{R}^+.
$$	
Clearly, $\varphi$ is nondecreasing.   Moreover, for all $t>0$, we have
$$
\varphi(t)\leq \frac{1}{2}t.
$$
Setting above $\varphi(t)$ instead of $t$ we get
$$
\varphi^2(t)\leq \frac{1}{2}\varphi(t)\leq \left(\frac{1}{2}\right)^2 t.
$$
Then, by induction, we get
$$
\varphi^n(t)\leq \left(\frac{1}{2}\right)^n t,\quad n\geq 0,\, t>0,
$$
which implies that $\displaystyle \sum_{n\geq 1} \varphi^n(t)<\infty$ for all $t>0$. Consequently, $\varphi\in \Phi$.
\end{example}

\begin{example}\label{ex2.4}
Let $0<\lambda_1<\lambda_2<1$. Consider the function
$$
\varphi(t)=\left\{\begin{array}{llll}
\arctan(\lambda_1t) &\mbox{if}& 0\leq t\leq \frac{1}{\lambda_1},\\[4pt]
\arctan(\lambda_2t) &\mbox{if}& 	t>\frac{1}{\lambda_1}.
\end{array}
\right.
$$
Clearly, $\varphi$ is nondecreasing. On the other hand, for all $t>0$, we have
$$
\varphi(t)\leq \max\{\lambda_1,\lambda_2\} t =\lambda_2 t,
$$
which yields
$$
\varphi^n(t)\leq (\lambda_2)^n t, \quad n\geq 0,\, t>0.
$$
Then  $\displaystyle \sum_{n\geq 1} \varphi^n(t)<\infty$ for all $t>0$, and $\varphi\in \Phi$.
Remark that in this example, the function $\varphi$ is not continuous.
\end{example}

In this section, we are concerned with the study of fixed points for the following class of single-valued mappings.

\begin{definition}\label{def2.5}
Let $d_i$, $i=1,2,3$,  be three metrics on $M$ with $|M|\geq 3$ and let $\varphi\in \Phi$. We denote by $\mathcal{F}(M,d_1,d_2,d_3,\varphi)$ the class of mappings $F: M\to M$ 	satisfying the three-points contraction
\begin{equation}\label{Contraction-Single}
d_1(Fx,Fy)+d_2(Fy,Fz)+ d_3(Fz,Fx)\leq \varphi\left(d_1(x,y)+d_2(y,z)+d_3(z,x)\right)	\end{equation}
for every three pairwise distinct points $x,y,z \in M$.
\end{definition}

\begin{remark}\label{RK}
Remark that, if $F: (M,d)\to (M,d)$ is a mapping contracting perimeters of triangles on $M$, in the sense of Definition \ref{def1.2}, then $F$ satisfies \eqref{Contraction-Single} with $d_i=d$, $i=1,2,3$, and $\varphi(s)=\lambda s$, $\lambda\in [0,1)$, that is, $F\in \mathcal{F}(M,d,d,d,\varphi)$.
\end{remark}

{The following example shows that $F\in \mathcal F (M, d_1, d_2, d_3, \varphi)$ can be discontinuous on $M$ with respect to one of the metrics $d_i$, $i\in \{1,2,3\}$.}

\begin{example}{
Let $M=[0,1]$, $d_1$ be the Euclidean distance on $M$,  and
$$
d_2(x,y)=d_3(x,y)=
\begin{cases}
    d_1(x,y), & x,y\in [0,\frac{1}{2}] \text { or } x,y\in (\frac{1}{2},1]; \\
    1, & \text{otherwise.}
\end{cases}
$$
The reader can easily verify that $d_1$ and $d_2$ are metrics on $M$. Let $F : M\to M$ be the mapping defined by
$$
Fx=
\begin{cases}
    \frac{1}{3}x, & x\in [0,\frac{1}{2}]; \\
    \frac{1}{2}x, &  x\in (\frac{1}{2},1].
\end{cases}
$$
It is clear that $F$ is discontinuous at $x=\frac{1}{2}$ as a mapping from  $(M,d_1)$ to $(M,d_1)$. For every three pairwise distinct points $x,y,z \in M$, let
$$
R(x,y,z)=\frac{d_1(Fx,Fy)+d_2(Fy,Fz)+ d_3(Fz,Fx)}{d_1(x,y)+d_2(y,z)+d_3(z,x)}.
$$
We claim that
\begin{equation}\label{eq-cl-phi}
R(x,y,z)\leq \frac{1}{2}	
\end{equation}
for every three pairwise distinct points $x,y,z \in M$. In order to show \eqref{eq-cl-phi}, it is sufficient to check the four cases:
\begin{itemize}
\item  $x,y,z\in [0,\frac{1}{2}]$,
\item  $x,y,z\in (\frac{1}{2},1]$,
\item  $x\in [0,\frac{1}{2}]$ and $y,z\in (\frac{1}{2},1]$,
\item $x,y\in [0,\frac{1}{2}]$ and $z\in (\frac{1}{2},1]$.
\end{itemize}
Case 1:  $x,y,z\in [0,\frac{1}{2}]$.  In this case, we have
$$
\begin{aligned}
R(x,y,z)&=\frac{d_1(\frac{x}{3},\frac{y}{3})+d_1(\frac{y}{3},\frac{z}{3})+ d_1(\frac{z}{3},\frac{x}{3})}{d_1(x,y)+d_1(y,z)+d_1(z,x)}\\
&=\frac{1}{3}.
\end{aligned}
$$
Case 2:  $x,y,z\in (\frac{1}{2},1]$. In this case, we have
$$
\begin{aligned}
R(x,y,z)&=\frac{d_1(\frac{x}{2},\frac{y}{2})+d_1(\frac{y}{2},\frac{z}{2})+ d_1(\frac{z}{2},\frac{x}{2})}{d_1(x,y)+d_1(y,z)+d_1(z,x)}\\
&=\frac{1}{2}.
\end{aligned}
$$
Case 3:  $x\in [0,\frac{1}{2}]$ and $y,z\in (\frac{1}{2},1]$. In this case, we have
\begin{equation}\label{ex-Petrov}
\begin{aligned}
R(x,y,z)&=\frac{d_1(\frac{x}{3},\frac{y}{2})+d_1(\frac{y}{2},\frac{z}{2})+ d_1(\frac{z}{2},\frac{x}{3})}{d_1(x,y)+d_2(y,z)+d_3(z,x)}\\
&= \frac{d_1(\frac{x}{3},\frac{y}{2})+d_1(\frac{y}{2},\frac{z}{2})+ d_1(\frac{z}{2},\frac{x}{3})}{d_1(x,y)+d_1(y,z)+1}\\
&=\frac{1}{2} \frac{d_1(\frac{2x}{3},y)+d_1(y,z)+ d_1(z,\frac{2x}{3})}{d_1(x,y)+d_1(y,z)+1}.
\end{aligned}
\end{equation}
On the other hand, we have
\begin{equation}\label{ex-Petrov-2}
\frac{d_1(\frac{2x}{3},y)+d_1(y,z)+ d_1(z,\frac{2x}{3})}{d_1(x,y)+d_1(y,z)+1}\leq 1.	
\end{equation}
Indeed, since
$$
x< y,\quad \frac{2x}{3}< y,\quad x < z,\quad \frac{2x}{3}< z,
$$
then \eqref{ex-Petrov-2} is equivalent to
$$
 \frac{x}{3}+(1-z)\geq 0.
$$
This inequality evidently holds since $x\geq 0$ and $z\leq 1$. Then, from \eqref{ex-Petrov} and \eqref{ex-Petrov-2}, we obtain
$$
R(x,y,z)\leq \frac{1}{2}.
$$
Case 4: $x,y\in [0,\frac{1}{2}]$ and $z\in (\frac{1}{2},1]$. In this case, we have
$$
\begin{aligned}
R(x,y,z)&=\frac{d_1(\frac{x}{3},\frac{y}{3})+d_2(\frac{y}{3},\frac{z}{2})+ d_3(\frac{z}{2},\frac{x}{3})}{d_1(x,y)+1+1}\\
&\leq \frac{\frac{1}{2}d_1(x,y)+\frac{1}{2}+\frac{1}{2}}{d_1(x,y)+1+1}\\
&=\frac{1}{2}\frac{d_1(x,y)+2}{d_1(x,y)+2}\\
&\leq \frac{1}{2}.
\end{aligned}
$$
From the above discussions, we deduce that \eqref{eq-cl-phi} holds for every three pairwise distinct points $x,y,z \in M$. This shows that $F\in \mathcal F (M, d_1, d_2, d_3, \varphi)$, where $\varphi(t)=\frac{t}{2}$ for every $t \in \mathbb{R}^+$.
}
\end{example}

Our main result in this section, is the following fixed point theorem.

\begin{theorem}\label{T2.7}
Let $d_i$, $i=1,2,3$,  be three metrics on $M$ such that $|M|\geq 3$ and $(M,d_1)$ be complete. Let $F: M\to M$ be a mapping satisfying    the following conditions:
\begin{itemize}
\item[{\rm{(i)}}] For all $u\in M$, $F(Fu)\neq u$, provided $Fu\neq u$;
\item[{\rm{(ii)}}] $F$ is continuous on $(M,d_1)$;
\item[{\rm{(iii)}}] $F\in \mathcal{F}(M,d_1,d_2,d_3,\varphi)$ for some $\varphi \in \Phi$.
\end{itemize}
Then $\Fix(F)\neq\emptyset$ and $|\Fix(F)|\leq 2$.
\end{theorem}

\begin{proof}
We first show that 	$\Fix(F)\neq\emptyset$. Let $u_0\in M$. Consider the Picard sequence $\{u_n\}\subset M$ defined by
$$
u_n=F^nu_0,\quad n\geq 0.
$$
If $u_{n-1}=u_{n}$ for some $n\geq 1$, then $u_{n-1}\in \Fix(F)$ and the theorem is proved. So, without restriction of the generality, we may assume that
$$
u_{n-1}\neq u_{n},\quad n\geq 1, 	
$$
which implies by (i) that
$$
u_{n-1}\neq u_{n+1},\quad n\geq 1.
$$
Consequently, for all $n\geq 1$, $u_{n-1},u_n$ and $u_{n+1}$ are three pairwise distinct points. By (iii), taking $(x,y,z)=(u_0,u_1,u_2)$ in \eqref{Contraction-Single}, we obtain
$$
d_1(Fu_0,Fu_1)+d_2(Fu_1,Fu_2)+ d_3(Fu_2,Fu_0)\leq \varphi\left(d_1(u_0,u_1)+d_2(u_1,u_2)+d_3(u_2,u_0)\right),	
$$
that is,
$$
d_1(u_1,u_2)+d_2(u_2,u_3)+ d_3(u_3,u_1)\leq \varphi\left(d_1(u_0,u_1)+d_2(u_1,u_2)+d_3(u_2,u_0)\right),	
$$
which implies by (C$_1$) that
\begin{equation}\label{S1}
\varphi\left(d_1(u_1,u_2)+d_2(u_2,u_3)+ d_3(u_3,u_1)\right)\leq \varphi^2\left(d_1(u_0,u_1)+d_2(u_1,u_2)+d_3(u_2,u_0)\right)	.
\end{equation}
Similarly, taking $(x,y,z)=(u_1,u_2,u_3)$ in \eqref{Contraction-Single}, we obtain
$$
d_1(u_2,u_3)+d_2(u_3,u_4)+ d_3(u_4,u_2)\leq \varphi\left(d_1(u_1,u_2)+d_2(u_2,u_3)+d_3(u_3,u_1)\right),	
$$
which implies by \eqref{S1} that
$$
d_1(u_2,u_3)+d_2(u_3,u_4)+ d_3(u_4,u_2)\leq \varphi^2\left(d_1(u_0,u_1)+d_2(u_1,u_2)+d_3(u_2,u_0)\right)	.
$$
Continuing in the same way, we obtain by induction that
$$
d_1(u_n,u_{n+1})+d_2(u_{n+1},u_{n+2})+ d_3(u_{n+2},u_n)\leq \varphi^n\left(d_1(u_0,u_1)+d_2(u_1,u_2)+d_3(u_2,u_0)\right)	
$$
for all $n\geq 0$, which yields
\begin{equation}\label{S2}
d_1(u_n,u_{n+1})\leq \varphi^n(\tau_0),\quad n\geq 1,		
\end{equation}
where
$$
\tau_0=d_1(u_0,u_1)+d_2(u_1,u_2)+d_3(u_2,u_0)>0.
$$
We now prove that $\{u_n\}$ is a Cauchy sequence on $(M,d_1)$.  For all $n,p\geq 1$, making use of the triangle inequality  and \eqref{S2}, we get
$$
\begin{aligned}
d_1(u_n,u_{n+p}) &\leq d_1(u_n,u_{n+1})+d_1(u_{n+1},u_{n+2})+\cdots+d_1(	u_{n+p-1},u_{n+p})\\
&\leq \varphi^n(\tau_0)+\varphi^{n+1}(\tau_0)+\cdots +\varphi^{n+p-1}(\tau_0)\\
&=\sum_{m=n}^{n+p-1}\varphi^{m}(\tau_0)\\
&=\sum_{m=0}^{n+p-1}	\varphi^{m}(\tau_0)-\sum_{m=0}^{n-1}\varphi^{m}(\tau_0),
\end{aligned}
$$
which implies by (C$_2$) that
$$
d_1(u_n,u_{n+p}) \leq \left(\sum_{m=0}^{\infty}	\varphi^{m}(\tau_0)-\sum_{m=0}^{n-1}\varphi^{m}(\tau_0)\right)\to 0\mbox{ as }n\to \infty.
$$
This proves that $\{u_n\}$ is a Cauchy sequence on the metric space $(M,d_1)$. Since $(M,d_1)$ is complete, there exists $u^*\in M$ such that
\begin{equation}\label{hot1}
\lim_{n\to \infty} d_1(u_n,u^*)=0,
\end{equation}
which implies by the   continuity of $F$ on $(M,d_1)$ that
\begin{equation}\label{hot2}
\lim_{n\to \infty} d_1(u_{n+1},Fu^*)=\lim_{n\to \infty} d_1(Fu_n,Fu^*)=0.
\end{equation}
Since every metric is continuous, by \eqref{hot1}, we have
\begin{equation}\label{hot3}
\lim_{n\to \infty} d_1(u_{n+1},Fu^*)=d_1(u^*,Fu^*).
\end{equation}
Then, by \eqref{hot2} and \eqref{hot3}, we get $u^*=Fu^*$, that is, $u^*\in \Fix(F)$.

We now show that $|\Fix(F)|\leq 2$. We argue by contradiction supposing that
$x,y$ and $z$ are three pairwise distinct fixed points of $F$. Then, making use of \eqref{Contraction-Single}
with the equalities $Fx=x$, $Fy=y$, $Fz=z$, we get
\begin{equation}\label{S3}{
d_1(x,y)+d_2(y,z)+ d_3(z,x)\leq \varphi\left(d_1(x,y)+d_2(y,z)+d_3(z,x)\right).	
}
\end{equation}
On the other hand, since
$$
d_1(x,y)+d_2(y,z)+d_3(z,x)>0,
$$
it follows from property \eqref{pp2} that
$$
\varphi\left(d_1(x,y)+d_2(y,z)+d_3(z,x)\right)<d_1(x,y)+d_2(y,z)+ d_3(z,x),
$$
which contradicts \eqref{S3}. This shows that $F$ admits at most two fixed points. The proof of Theorem \ref{T2.7} is then completed.
\end{proof}

We now study some particular cases of Theorem \ref{T2.7}. Recall that for a given metric space $X$, a point $x \in X$ is said to be an \emph{accumulation point} of  $X$ if every open ball centered at $x$ contains infinitely many points of $X$.

\begin{proposition}\label{PR2.8}
Let $d_i$, $i=1,2,3$,  be three metrics on $M$ such that $|M|\geq 3$ and
\begin{equation}\label{cd-d1+d2-PR2.7}
\max\{d_2(u,v),d_3(u,v)\}\leq \kappa d_1(u,v),\quad  u,v\in M	
\end{equation}
for some constant $\kappa>0$. If  $F\in \mathcal{F}(M,d_1,d_2,d_3,\varphi)$ for some $\varphi\in \Phi$, then $F$ is continuous on $(M,d_1)$.
\end{proposition}

\begin{proof}

Let us show that $F$ is continuous at every point $z_0\in M$.

\noindent $\bullet$ Case 1: $z_0$ is an isolated point in $(M,d_1)$.\\
In this case,  $F$ 	is obviously continuous at $z_0$.

\noindent $\bullet$  Case 2:  $z_0$ is an accumulation point in $(M,d_1)$.\\
For all $\varepsilon>0$, let
\begin{equation}\label{delta-eps}
\delta_\varepsilon=\frac{\varepsilon}{1+3\kappa}.
\end{equation}
Since $z_0$ is an accumulation point in $(M,d_1)$, there exist points  $y,z\in M$ such that $y\neq z$,
\begin{equation}\label{S2-PR2.7}
0<d_1(z_0,y)< \delta_\varepsilon
\, \text{ and } \,
0<d_1(z_0,z)< \delta_\varepsilon.
\end{equation}
Remark that from \eqref{S2-PR2.7}, $z_0,z$ and $y$ are three pairwise distinct points. Then, making use of \eqref{Contraction-Single}, we obtain
$$
d_1(Fz_0,Fz)+d_2(Fz,Fy)+ d_3(Fy,Fz_0)\leq \varphi\left(d_1(z_0,z)+d_2(z,y)+d_3(y,z_0)\right),
$$
which implies that
\begin{equation}\label{S3-PR2.7}
d_1(Fz_0,Fz)\leq \varphi\left(d_1(z_0,z)+d_2(z,y)+d_3(y,z_0)\right).
\end{equation}
Since $d_1(z_0,z)+d_2(z,y)+d_3(y,z_0)>0$, it follows from property \eqref{pp2} that
$$
\varphi\left(d_1(z_0,z)+d_2(z,y)+d_3(y,z_0)\right)<d_1(z_0,z)+d_2(z,y)+d_3(y,z_0),
$$
which implies by \eqref{cd-d1+d2-PR2.7} and the triangle inequality that
$$
\begin{aligned}
\varphi\left(d_1(z_0,z)+d_2(z,y)+d_3(y,z_0)\right) &< d_1(z_0,z)+\kappa d_1(z,y)+\kappa d_1(y,z_0)\\
&\leq d_1(z_0,z)+\kappa d_1(z,z_0)+\kappa d_1(z_0,y)+\kappa d_1(y,z_0).
\end{aligned}
$$
Then, from \eqref{delta-eps} and \eqref{S2-PR2.7}, we deduce that
\begin{equation}\label{S4-PR2.7}
\varphi\left(d_1(z_0,z)+d_2(z,y)+d_3(y,z_0)\right) 	
<\delta_{\varepsilon}(1+3\kappa)=\varepsilon.
\end{equation}
Finally, by \eqref{S3-PR2.7} and \eqref{S4-PR2.7}, we get
$$
d_1(Fz_0,Fz)<\varepsilon,
$$
which proves that $F$ is continuous at $z_0$. This completes the proof of Proposition \ref{PR2.8}.
\end{proof}

\begin{remark}\label{RK2}
Having traced the proofs of Theorem~\ref{T2.7} and Proposition~\ref{PR2.8} we see that nowhere properties of metrics $d_2$ and $d_3$, such as triangle inequalities, continuity, etc., have been used. Thus, Theorem~\ref{T2.7} and Proposition~\ref{PR2.8} are valid under the assumption that $d_2$ and $d_3$ are semimetrics. Recall that semimetric is a function  $d\colon X\times X\to \mathbb{R}^+$ satisfying for all $x,y \in X$ only two axioms of metric space:
 \begin{enumerate}
   \item $(d(x,y)=0)\Leftrightarrow (x=y)$,
   \item $d(x,y)=d(y,x)$.
 \end{enumerate}
A pair $(X,d)$, where  $d$  is a semimetric on $X$, is called a \emph{semimetric space}. Such spaces were first examined by Fr\'{e}chet in~\cite{Fr06}, where he called them ``classes (E)''. Later these spaces attracted the attention of many mathematicians.
\end{remark}

From Theorem \ref{T2.7} and Proposition \ref{PR2.8}, we deduce the following result.

\begin{corollary}\label{CR2.9}
Let $d_i$, $i=1,2,3$,  be three metrics on $M$ such that $|M|\geq 3$, $(M,d_1)$ is complete and \eqref{cd-d1+d2-PR2.7} holds for some constant $\kappa>0$. 	Let $F: M\to M$ be a mapping satisfying    the following conditions:
\begin{itemize}
\item[{\rm{(i)}}] For all $u\in M$, $F(Fu)\neq u$, provided $Fu\neq u$;
\item[{\rm{(ii)}}] $F\in \mathcal{F}(M,d_1,d_2,d_3,\varphi)$ for some $\varphi \in \Phi$.
\end{itemize}
Then $\Fix(F)\neq\emptyset$ and $|\Fix(F)|\leq 2$.
\end{corollary}

\begin{remark}
Taking in Corollary \ref{CR2.9} $d_i=d$, $i=1,2,3$, $\kappa=1$ and $\varphi(t)=\lambda t$, $\lambda\in [0,1)$, we obtain Theorem~\ref{T1.3} (see Remark \ref{RK}).
\end{remark}

We provide below an example illustrating Theorem \ref{T2.7}.

\begin{example}\label{ex2.11}
Let $M=\{w_1,w_2,w_3,w_4\}\subset \mathbb{R}^2$, where $$
w_1=\left(-\frac{9}{4},0\right),\,\,w_2=(0,0),\,\, w_3=\left(\frac{175}{72},-\sqrt{9-\left(\frac{175}{72}\right)^2}\right),\,\, w_4=\left(-\frac{55}{24},\frac{5\sqrt{23}}{24}\right).
$$	
 Consider the mapping $F: M\to M$ defined by
$$
Fw_1=w_1,\,\, Fw_2=w_2,\,\,Fw_3=w_4,\,\, Fw_4=w_1.
$$
Let $\delta$ be the discrete metric on $M$, that is,
$$
\delta(w_i,w_j)=\left\{\begin{array}{llll}
0 &\mbox{if}& i=j,\\[4pt]
1 &\mbox{if}& i\neq j.	
\end{array}
\right.
$$
Remark that $F$ is not a mapping contracting perimeters of triangles on $(M,\delta)$ (in the sense of Definition \ref{def1.2} with $d=\delta$). Indeed, we have
$$
\begin{aligned}
\frac{\delta(Fw_1,Fw_2)+\delta(Fw_2,Fw_3)+\delta(Fw_3,Fw_1)}{\delta(w_1,w_2)+\delta(w_2,w_3)+\delta(w_3,w_1)}
&=\frac{\delta(w_1,w_2)+\delta(w_2,w_4)+\delta(w_4,w_1)}{\delta(w_1,w_2)+\delta(w_2,w_3)+\delta(w_3,w_1)}\\
&=1.
\end{aligned}
$$

We now consider the three metrics $d_1,d_2,d_3$ on $M$, where $d_1=\delta$ and

\begin{equation}\label{d2=d3-ex2.11}
d_2(w_i,w_j)=d_3(w_i,w_j)=\|w_i-w_j\|,\quad i,j\in\{1,2,3,4\}.
\end{equation}
Here, $\|\cdot\|$ denotes the Euclidean norm. Elementary calculations show that
$$
\|w_i-w_j\|=\left\{\begin{array}{lllll}
\frac{9}{4} &\mbox{if}& (i,j)=(1,2),\\[4pt]
5 	&\mbox{if}& (i,j)=(1,3),\\[4pt]
1 &\mbox{if}& (i,j)=(1,4),\\[4pt]
3 &\mbox{if}& (i,j)=(2,3),\\[4pt]
\frac{5}{2} &\mbox{if}& (i,j)=(2,4),\\[4pt]
5.46846&\mbox{if}& (i,j)=(3,4).
\end{array}
\right.	
$$

Clearly, $(M,d_1)$ is a complete metric space and $F$ satisfies conditions (i) and (ii) of Theorem \ref{T2.7}. 	 We claim that
\begin{equation}\label{claim-ex2.11}
d_1(Fw_i,Fw_j)+d_2(Fw_j,Fw_k)+d_3(Fw_k,Fw_i)\leq \frac{23}{25}\left(d_1(w_i,w_j)+d_2(w_j,w_k)+d_3(w_k,w_i)\right)
\end{equation}
for every three pairwise distinct indices $i,j,k\in\{1,2,3,4\}$.

\begin{table}[htt!]
\begin{center}
\begin{tabular}{|c | c | c |c|}
  \hline			
  $(i,j,k)$ & $A(i,j,k)$  & $B(i,j,k)$ & $R(i,j,k)$\\
  \hline
  $(1,2,3)$ & 9/2 & 9 & 1/2 \\
  \hline
  $(1,3,2)$ & 23/4 & 25/4 & 23/25 \\
  \hline
 $(2,3,1)$ & 17/4 & 33/4 & 17/33 \\
  \hline
  $(1,2,4)$ & 13/4 & 9/2 & 13/18 \\
  \hline
  $(1,4,2)$ & 9/2 & 23/4 & 18/23 \\
  \hline
  $(2,4,1)$ & 13/4 & 27/4 & 13/27\\
  \hline
  $(1,3,4)$ & 2 & 6.46846 & 0.3091\\
  \hline
  $(1,4,3)$ & 2 & 11. 46846 & 0.1743\\
  \hline
  $(3,4,1)$ & 2 & 7 & 2/7\\
  \hline
  $(2,3,4)$ & 17/4 & 8.96846 & 0.4738 \\
  \hline
  $(2,4,3)$ & 9/2 & 8.46846 & 0.5313\\
  \hline
  $(3,4,2)$ & 23/4 & 13/2 & 23/26\\
  \hline
\end{tabular}
\caption{The values of $A(i,j,k)$, $B(i,j,k)$ and  $R(i,j,k)$}\label{Tb1}
\end{center}
\end{table}

For every three pairwise distinct indices $i,j,k\in\{1,2,3,4\}$,  let
$$
R(i,j,k)=\frac{d_1(Fw_i,Fw_j)+d_2(Fw_j,Fw_k)+d_3(Fw_k,Fw_i)}{d_1(w_i,w_j)+d_2(w_j,w_k)+d_3(w_k,w_i)}=\frac{A(i,j,k)}{B(i,j,k)}.
$$
Remark that by~\eqref{d2=d3-ex2.11}
$$
R(i,j,k)=R(j,i,k).
$$
So, for every three pairwise distinct indices  $i,j,k\in\{1,2,3,4\}$, we have just to show that \eqref{claim-ex2.11} holds for $(i,j,k)$, $(i,k,j)$,  $(j,k,i)$. Namely, we have to check twelve cases. Table \ref{Tb1} provides the different values of  $A(i,j,k)$,  $B(i,j,k)$ and $R(i,j,k)$, which confirm \eqref{claim-ex2.11}.

Consequently, $F\in \mathcal{F}(M,d_1,d_2,d_3,\varphi)$ with $\varphi(t)=\frac{23}{25}t$ for all $t\in \mathbb{R}^+$.  Then Theorem \ref{T2.7} applies. On the other hand, observe that $\Fix(F)=\{w_1,w_2\}$, which confirms the result given by Theorem \ref{T2.7}.
\end{example}

\section{Three-points multi-valued contractions} \label{sec3}

In this section,  we are concerned with the study of fixed points for  the following class of multi-valued mappings.

\begin{definition}\label{def3.1}
Let  $(M,d)$ be a metric space with $|M|\geq 3$ and $\lambda\in (0,1)$.  We denote by $\widetilde{\mathcal{F}}(M,d,\lambda)$,  the class of multi-valued mappings $F: M\to  CB(M)$ 	satisfying the three-points multi-valued contraction
\begin{equation}\label{Contraction-Multi}
H(Fx,Fy)+H(Fy,Fz)+ \mathcal{D}(Fz,Fx)\leq \lambda \left(d(x,y)+d(y,z)+d(z,x)\right)	
\end{equation}
for every three pairwise distinct points $x,y,z \in M$.
\end{definition}

The following lemma (see  \cite{Nadler}) will be used later.

\begin{lemma}\label{L3.2}
Let $(M,d)$ be a metric space and  $A,B\in CB(M)$. Then, for all $a\in A$ and $\varepsilon>0$, there exists $b\in B$ such that
$$
d(a,b)\leq H(A,B)+\varepsilon.
$$
\end{lemma}

We first establish the following result.

\begin{proposition}\label{PR3.3}
Let  $(M,d)$ be a metric space with $|M|\geq 3$ and $\lambda\in (0,1)$. If $F\in \widetilde{\mathcal{F}}(M,d,\lambda)$, then 	$F: (M,d)\to  (CB(M),H)$ is continuous. 	
\end{proposition}

\begin{proof}
Let $z_0\in M$. We distinguish two cases.\\
$\bullet$ Case 1: $z_0$ is an isolated point in $(M,d_1)$. \\
In this case, $F$ is obviously continuous at $z_0$.  \\
$\bullet$ Case 2: $z_0$ is an accumulation point in $(M,d_1)$.\\
For all $\varepsilon>0$, let
\begin{equation}\label{S1-PR3.3}
\delta_\varepsilon=\frac{\varepsilon}{4\lambda}. 	
\end{equation}
Since $z_0$ is an accumulation point in $(M,d_1)$, there exist points  $y,z\in M$ such that $y\neq z$,
\begin{equation}\label{S2-PR3.3}
0<d_1(z_0,y)< \delta_\varepsilon
\, \text{ and } \,
0<d_1(z_0,z)< \delta_\varepsilon.
\end{equation}
Then, by \eqref{S2-PR3.3} $z_0,z$ and $y$ are three pairwise distinct points. Hence, making use of \eqref{Contraction-Multi}, we obtain
$$
H(Fz_0,Fz)+H(Fz,Fy)+ \mathcal{D}(Fy,Fz_0)\leq \lambda \left(d(z_0,z)+d(z,y)+d(y,z_0)\right),
$$
which implies by the triangle inequality, \eqref{S2-PR3.3} and \eqref{S1-PR3.3}  that
$$
\begin{aligned}
H(Fz_0,Fz) &\leq \lambda \left(d(z_0,z)+d(z,y)+d(y,z_0)\right)\\
&\leq \lambda \left(d(z_0,z)+d(z,z_0)+d(z_0,y)+d(y,z_0)\right)\\
&=2\lambda \left(d(z_0,z)+d(z_0,y)\right)\\
&<  2\lambda \left(\delta_\varepsilon+\delta_\varepsilon\right)\\
&= 4\lambda \delta_\varepsilon \\
&=\varepsilon.
\end{aligned}
$$
This shows the continuity of $F$ at $z_0$.
\end{proof}

Our main result in this section, is the following fixed point theorem.

\begin{theorem}\label{T3.4}
Let $(M,d)$ be a complete metric space with $|M|\geq 3$. 	Let $F: M\to CB(M)$ be a multi-valued mapping satisfying    the following conditions:
\begin{itemize}
\item[{\rm{(i)}}] For all $u,v\in M$, we have
$$
v\in Fu,\, u\neq v\implies u\not\in Fv;
$$	
\item[{\rm{(ii)}}] $F\in \widetilde{\mathcal{F}}(M,d,\lambda)$ for some $\lambda\in (0,1)$.
\end{itemize}
Then $\Fix(F)\neq \emptyset$.
\end{theorem}

\begin{proof}
Let $u_0\in M$ and $u_1\in Fu_0$.  By Lemma \ref{L3.2}, there exists $u_2\in Fu_1$ such that
$$
d(u_1,u_2)\leq H(Fu_0,Fu_1)+\lambda,
$$
and there exists $u_3\in Fu_2$ such that
$$
d(u_2,u_3)\leq H(Fu_1,Fu_2)+\lambda^2.
$$	
Continuing in the same way, 	by induction, we construct a sequence $\{u_n\}\subset M$ such that
\begin{equation}\label{SS1}
u_{n+1}\in Fu_n,\quad n\geq 0
\end{equation}
and
\begin{equation}\label{SS2}
d(u_n,u_{n+1})\leq H(Fu_{n-1},Fu_n)+\lambda^n,\quad n\geq 1.
\end{equation}
If $u_{n}=u_{n+1}$ for some $n\geq 0$, then by \eqref{SS1}, $u_n\in \Fix(F)$, and the theorem is proved. Then, without restriction of the generality, we may suppose that $u_n\neq u_{n+1}$ for all $n\geq 0$, which implies by \eqref{SS1} and (i) that $u_n\not\in Fu_{n+1}$.  Since $u_{n+2}\in Fu_{n+1}$, then $u_n\neq u_{n+2}$. Consequently, for all $n\geq 0$, $u_n,u_{n+1}$ and $u_{n+2}$ are three pairwise distinct points in $M$. On the other hand, by \eqref{SS2}, we have
\begin{equation}\label{SS3}
d(u_{n-1},u_{n})\leq H(Fu_{n-2},Fu_{n-1})+\lambda^{n-1},\quad n\geq 2.
\end{equation}
Furthermore, by the definition of the diameter, we have
\begin{equation}\label{SS4}
d(u_{n+1},u_{n-1})\leq \mathcal{D}(Fu_n,Fu_{n-2}),\quad n\geq 2. 	
\end{equation}
Then, it follows from \eqref{SS2}, \eqref{SS3} and \eqref{SS4} that
\begin{equation}\label{SS5}
\begin{aligned}
& d(u_{n-1},u_{n})+d(u_n,u_{n+1})+d(u_{n+1},u_{n-1})\\
&\leq 	H(Fu_{n-2},Fu_{n-1})+ H(Fu_{n-1},Fu_n)+\mathcal{D}(Fu_n,Fu_{n-2})+\lambda^{n-1}+\lambda^n,\quad n\geq 2.
\end{aligned}
\end{equation}
Taking into consideration that for all $n\geq 0$, $u_n,u_{n+1}$ and $u_{n+2}$ are three pairwise distinct points and making use of   \eqref{Contraction-Multi} with $(x,y,z)=(u_{n-2},u_{n-1},u_n)$, we obtain
\begin{equation}\label{SS6}
\begin{aligned}
&H(Fu_{n-2},Fu_{n-1})+ H(Fu_{n-1},Fu_n)+\mathcal{D}(Fu_n,Fu_{n-2})\\
&\leq \lambda \left(d(u_{n-2},u_{n-1})+d(u_{n-1},u_n)+d(u_n,u_{n-2})\right),\quad n\geq 2.
\end{aligned}	
\end{equation}
Then, from \eqref{SS5} and \eqref{SS6}, we deduce that
$$
\begin{aligned}
&d(u_{n-1},u_{n})+d(u_n,u_{n+1})+d(u_{n+1},u_{n-1})\\
&\leq \lambda \left(d(u_{n-2},u_{n-1})+d(u_{n-1},u_n)+d(u_n,u_{n-2})\right)+	\lambda^{n-1}+\lambda^n,\quad n\geq 2,
\end{aligned}
$$
that is,
\begin{equation}\label{SS7}
p_n\leq \lambda p_{n-1}+\lambda^n+\lambda^{n+1},\quad n\geq 1,
\end{equation}
where
\begin{equation}\label{pn}
p_n=d(u_{n},u_{n+1})+d(u_{n+1},u_{n+2})+d(u_{n+2},u_{n}),\quad n\geq 0.
\end{equation}
From \eqref{SS7}, we get
\begin{eqnarray*}
&& p_1 \leq \lambda p_0+\lambda+\lambda^2,\\
&& p_2\leq \lambda p_1+\lambda^2 +\lambda^3\leq \lambda^2p_0+2(\lambda^2+\lambda^3),\\
&&p_3\leq \lambda p_2+\lambda^3+\lambda^4\leq \lambda^3p_0+3(\lambda^3+\lambda^4),\\
&&\vdots \\
&&p_n\leq \lambda^np_0+n(\lambda^n+\lambda^{n+1}),\quad n\geq 0,	
\end{eqnarray*}
which implies by \eqref{pn} that
\begin{equation}\label{SS8}
d(u_{n},u_{n+1})\leq 	\lambda^np_0+n(\lambda^n+\lambda^{n+1}),\quad n\geq 0.
\end{equation}
Since $0<\lambda<1$ we get $\lambda^{n+1}<\lambda^{n}$. Hence, it follows from~\eqref{SS8} that
\begin{equation}\label{SS88}
d(u_{n},u_{n+1})< \lambda^np_0+n(\lambda^n+\lambda^{n})=
\lambda^n(p_0+2n)
,\quad n\geq 0.
\end{equation}
We now show that $\{u_n\}$ is a Cauchy sequence on $(M,d)$. Indeed, making use of \eqref{SS88} and the triangle inequality, for all $n,k\geq 1$,  we obtain
\begin{equation}\label{SS9}
\begin{aligned}
d(u_n,u_{n+k})&\leq \sum_{m=n}^{n+k-1} d(u_m,u_{m+1})\\
&< p_0 \sum_{m=n}^{n+k-1} \lambda^m+2\sum_{m=n}^{n+k-1}m\lambda^m\\
&= p_0\left(\sum_{m=0}^{n+k-1} \lambda^m-\sum_{m=0}^{n-1}\lambda^m\right)+2\left(\sum_{m=0}^{n+k-1}m\lambda^m-\sum_{m=0}^{n-1}m\lambda^m\right).
\end{aligned}
\end{equation}
On the other hand, since $0<\lambda<1$,  the two series $\sum_{m\geq 0}\lambda^m$ and $\sum_{m\geq 0}m\lambda^m$ are convergent. Then, by \eqref{SS9}, we obtain
$$
\begin{aligned}
d(u_n,u_{n+k}) < 	p_0\left(\sum_{m=0}^{\infty} \lambda^m-\sum_{m=0}^{n-1}\lambda^m\right)+2\left(\sum_{m=0}^{\infty}m\lambda^m-\sum_{m=0}^{n-1}m\lambda^m\right)
\to 0 \mbox{ as }n\to \infty,
\end{aligned}
$$
which shows that $\{u_n\}$ is a Cauchy sequence on $(M,d)$. Then, from the completeness of $(M,d)$, we deduce that there exists $u^*\in M$ such that
\begin{equation}\label{SS10}
\lim_{n\to \infty}d(u_n,u^*)=0. 	
\end{equation}
We now prove that $u^*$ is a fixed point of $F$. Since $F: (M,d)\to  (CB(M),H)$ is continuous (by (ii) and Proposition \ref{PR3.3}), we deduce from \eqref{SS10} that
\begin{equation}\label{SS11}
\lim_{n\to \infty}H(Fu_n,Fu^*)=0. 	
\end{equation}
Then, making use of \eqref{SS1}, \eqref{SS10} and \eqref{SS11}, we get
$$
\begin{aligned}
D(u^*,Fu^*) &\leq d(u^*,u_{n+1})+D(u_{n+1},Fu^*)\\
&\leq \left(d(u^*,u_{n+1})+H(Fu_n,Fu^*)\right)\to 0\mbox{ as }n\to \infty.	
\end{aligned}
$$
Consequently, we obtain $D(u^*,Fu^*)=0$. Since $Fu^*$ is closed, we deduce  that $u^*\in Fu^*$, that is, $u^*\in \Fix(F)$. The proof of Theorem \ref{T3.4} is then completed.
\end{proof}

\begin{definition}\label{def3.5}
Let  $(M,d)$ be a metric space with $|M|\geq 3$ and $\lambda\in (0,1)$.  We denote by $\widetilde{\mathcal{F}}'(M,d,\lambda)$ and $\widetilde{\mathcal{F}}''(M,d,\lambda)$  the classes of multi-valued mappings $F: M\to  CB(M)$ 	satisfying the three-points multi-valued contractions
\begin{equation*}
H(Fx,Fy)+\mathcal{D}(Fy,Fz)+ \mathcal{D}(Fz,Fx)\leq \lambda \left(d(x,y)+d(y,z)+d(z,x)\right),
\end{equation*}
\begin{equation*}
\mathcal{D}(Fx,Fy)+\mathcal{D}(Fy,Fz)+ \mathcal{D}(Fz,Fx)\leq \lambda \left(d(x,y)+d(y,z)+d(z,x)\right)	
\end{equation*}
for every three pairwise distinct points $x,y,z \in M$, respectively.
\end{definition}
Since the inequalities
$$
H(Fx,Fy)+H(Fy,Fz)+ \mathcal{D}(Fz,Fx)
$$
$$
\leq
H(Fx,Fy)+\mathcal{D}(Fy,Fz)+ \mathcal{D}(Fz,Fx),
$$
$$
\leq \mathcal{D}(Fx,Fy)+\mathcal{D}(Fy,Fz)+ \mathcal{D}(Fz,Fx),
$$
hold, we obtain the inclusions $\widetilde{\mathcal{F}}''(M,d,\lambda)\subseteq \widetilde{\mathcal{F}}'(M,d,\lambda) \subseteq \widetilde{\mathcal{F}}(M,d,\lambda)$. Hence, we get the following.
\begin{corollary}
Theorem~\ref{T3.4} holds for the classes $\widetilde{\mathcal{F}}'(M,d,\lambda)$ and $\widetilde{\mathcal{F}}''(M,d,\lambda)$.
\end{corollary}

We give below an example to illustrate Theorem \ref{T3.4}.

\begin{example}\label{exx3.5}
Let $M=\{v_1,v_2,v_3\}$ and $d$ be the discrete metric on $M$, that is,
$$
d(v_i,v_j)=\left\{\begin{array}{llll}
0 &\mbox{if}& i=j,\\[4pt]
1 &\mbox{if}& i\neq j.	
\end{array}
\right.
$$	
Consider the multi-valued mapping $F: (M,d)\to (CB(M),H)$ defined by
$$
Fv_1=\{v_1\},\,\, Fv_2=\{v_1\},\,\, Fv_3=\{v_1,v_3\},
$$
where $H$ is the Hausdorff-Pompeiu metric on  $CB(M)$ induced by $d$.

Observe that $F$ satisfies condition (i) of Theorem \ref{T3.4}. Indeed, we have
$$
v_1\in Fv_2,\,\, v_1\neq v_2,\,\, v_2\not\in Fv_1=\{v_1\}
$$
and
$$
v_1\in Fv_3,\,\, v_1\neq v_3,\,\, v_3\not\in Fv_1=\{v_1\}.
$$
On the other hand, for $(i,j,k)=(1,2,3)$, we have
$$
\begin{aligned}
&H(Fv_i,Fv_j)+H(Fv_j,Fv_k)+\mathcal{D}(Fv_k,Fv_i)	\\
&=H\left(\{v_1\},\{v_1\}\right)+H\left(\{v_1\},\{v_1,v_3\}\right)+\mathcal{D}\left(\{v_1,v_3\},\{v_1\}\right)\\
&=0+1+1\\
&=2
\end{aligned}
$$
and
$$
d(v_i,v_j)+d(v_j,v_k)+d(v_k,v_i)=3,
$$
which show that
\begin{equation}\label{Ratio}
\frac{H(Fv_i,Fv_j)+H(Fv_j,Fv_k)+\mathcal{D}(Fv_k,Fv_i)}{d(v_i,v_j)+d(v_j,v_k)+d(v_k,v_i)}=\frac{2}{3}.
\end{equation}
Similar calculations show that for every three pairwise distinct indices $i,j,k\in \{1,2,3\}$, \eqref{Ratio} holds. Consequently, $F\in \widetilde{\mathcal{F}}(M,d,\lambda)$ for every $\frac{2}{3}\leq \lambda<1$.  Then $F$ satisfies also condition (ii) of Theorem \ref{T3.4}. Furthermore, we have
$$
\Fix(F)=\{v_1,v_3\},
$$
which confirms that $\Fix(F)\neq \emptyset$.

Observe also that in this example, Nadler's fixed point theorem (Theorem \ref{T1.1}) is inapplicable. This can be easily seen remarking that
$$
\frac{H(Fv_1,Fv_3)}{d(v_1,v_3)}=H(Fv_1,Fv_3)=H\left(\{v_1\},\{v_1,v_3\}\right)=1.
$$
\end{example}

\begin{definition}\label{def3.8}
Let  $(M,d)$ be a metric space with $|M|\geq 3$ and $\lambda\in (0,\frac{1}{2})$.  We denote by $\overline{\mathcal{F}}(M,d,\lambda)$,  the class of multi-valued mappings $F: M\to  CB(M)$ 	satisfying the three-points multi-valued contraction
\begin{equation}\label{Contraction-MultiH}
H(Fx,Fy)+H(Fy,Fz)+H(Fz,Fx)\leq \lambda \left(d(x,y)+d(y,z)+d(z,x)\right)	
\end{equation}
for every three pairwise distinct points $x,y,z \in M$.
\end{definition}

Similarly to Proposition~\ref{PR3.3} we establish the following.

\begin{proposition}\label{PR3.3H}
Let  $(M,d)$ be a metric space with $|M|\geq 3$ and $\lambda\in (0,\frac{1}{2})$. If $F\in \overline{\mathcal{F}}(M,d,\lambda)$, then 	$F: (M,d)\to  (CB(M),H)$ is continuous. 	
\end{proposition}

\begin{theorem}\label{T3.4H}
Let $(M,d)$ be a complete metric space with $|M|\geq 3$. 	Let $F: M\to CB(M)$ be a multi-valued mapping satisfying    the following conditions:
\begin{itemize}
\item[{\rm{(i)}}] For all $u,v\in M$, we have
$$
v\in Fu,\, u\neq v\implies u\not\in Fv;
$$	
\item[{\rm{(ii)}}] $F\in \overline{\mathcal{F}}(M,d,\lambda)$ for some $\lambda\in (0,\frac{1}{2})$.
\end{itemize}
Then $\Fix(F)\neq \emptyset$.
\end{theorem}

\begin{proof}
The beginning of the proof of this theorem repeats word for word the proof of the Theorem~\ref{T3.4} up to inequality~\eqref{SS3}.

Further, it follows from the triangle inequality, \eqref{SS2} and \eqref{SS3}  that
\begin{equation}\label{SS5H}
\begin{aligned}
& d(u_{n-1},u_{n})+d(u_n,u_{n+1})+d(u_{n+1},u_{n-1})\\
& \leqslant d(u_{n-1},u_{n})+d(u_n,u_{n+1})+d(u_{n-1},u_{n})+d(u_n,u_{n+1})\\
& =2(d(u_{n-1},u_{n})+d(u_n,u_{n+1}))\\
&\leq 	2( H(Fu_{n-2},Fu_{n-1})+ H(Fu_{n-1},Fu_n))+2\lambda^{n-1}+2\lambda^n,\quad n\geq 2.
\end{aligned}
\end{equation}
Taking into consideration that for all $n\geq 0$, $u_n,u_{n+1}$ and $u_{n+2}$ are three pairwise distinct points and making use of   \eqref{Contraction-MultiH} with $(x,y,z)=(u_{n-2},u_{n-1},u_n)$, we obtain
\begin{equation*}
\begin{aligned}
&H(Fu_{n-2},Fu_{n-1})+ H(Fu_{n-1},Fu_n)+H(Fu_n,Fu_{n-2})\\
&\leq \lambda \left(d(u_{n-2},u_{n-1})+d(u_{n-1},u_n)+d(u_n,u_{n-2})\right),\quad n\geq 2.
\end{aligned}	
\end{equation*}
and
\begin{equation}\label{SS6H}
\begin{aligned}
&H(Fu_{n-2},Fu_{n-1})+ H(Fu_{n-1},Fu_n)\\
&\leq \lambda \left(d(u_{n-2},u_{n-1})+d(u_{n-1},u_n)+d(u_n,u_{n-2})\right),\quad n\geq 2.
\end{aligned}	
\end{equation}
Then, from \eqref{SS5H} and \eqref{SS6H}, we deduce that
$$
\begin{aligned}
&d(u_{n-1},u_{n})+d(u_n,u_{n+1})+d(u_{n+1},u_{n-1})\\
&\leq 2\lambda \left(d(u_{n-2},u_{n-1})+d(u_{n-1},u_n)+d(u_n,u_{n-2})\right)+	2\lambda^{n-1}+2\lambda^n,\quad n\geq 2,
\end{aligned}
$$
that is,
\begin{equation}\label{SS7H}
p_n\leq 2\lambda p_{n-1}+2\lambda^n+2\lambda^{n+1},\quad n\geq 1,
\end{equation}
where
\begin{equation}\label{pnH}
p_n=d(u_{n},u_{n+1})+d(u_{n+1},u_{n+2})+d(u_{n+2},u_{n}),\quad n\geq 0.
\end{equation}
From \eqref{SS7H}, we get
\begin{eqnarray*}
&& p_1 \leq 2\lambda p_0+2\lambda+2\lambda^2,\\
&& p_2\leq 2\lambda p_1+2\lambda^2 +2\lambda^3\leq 4\lambda^2p_0+6(\lambda^2+\lambda^3),\\
&&p_3\leq 2\lambda p_2+2\lambda^3+2\lambda^4\leq 8\lambda^3p_0+14(\lambda^3+\lambda^4),\\
&&\vdots \\
&&p_n\leq (2\lambda)^np_0+2(2^n-1)(\lambda^n+\lambda^{n+1}),\quad n\geq 0,	
\end{eqnarray*}
which implies by \eqref{pnH} that
\begin{equation}\label{SS8H}
d(u_{n},u_{n+1})\leq 	(2\lambda)^np_0+2(2^n-1)(\lambda^n+\lambda^{n+1}),\quad n\geq 0.
\end{equation}
Since $0<\lambda<\frac{1}{2}$ we get $\lambda^{n+1}<\lambda^{n}$. Hence, it follows from~\eqref{SS8} that
\begin{equation}\label{SS88H}
d(u_{n},u_{n+1})< (2\lambda)^np_0+2(2^n-1)(\lambda^n+\lambda^{n})=
(2\lambda)^np_0+4(2^n-1)\lambda^n
,\quad n\geq 0.
\end{equation}
We now show that $\{u_n\}$ is a Cauchy sequence on $(M,d)$. Indeed, making use of \eqref{SS88} and the triangle inequality, for all $n,k\geq 1$,  we obtain
\begin{equation}\label{SS9H}
\begin{aligned}
d(u_n,u_{n+k})&\leq \sum_{m=n}^{n+k-1} d(u_m,u_{m+1})\\
&< p_0 \sum_{m=n}^{n+k-1} (2\lambda)^m+4\sum_{m=n}^{n+k-1}(2^m-1)\lambda^m\\
&= p_0\left(\sum_{m=0}^{n+k-1} (2\lambda)^m-\sum_{m=0}^{n-1}(2\lambda)^m\right)+4\left(\sum_{m=0}^{n+k-1}(2^m-1)\lambda^m-\sum_{m=0}^{n-1}(2^m-1)\lambda^m\right).
\end{aligned}
\end{equation}
On the other hand, since $0<\lambda<\frac{1}{2}$,  the two series $\sum_{m\geq 0}(2\lambda)^m$ and $\sum_{m\geq 0}(2^m-1)\lambda^m$ are convergent. Then, by \eqref{SS9H}, we obtain
$$
\begin{aligned}
d(u_n,u_{n+k}) &< p_0\left(\sum_{m=0}^{\infty} (2\lambda)^m-\sum_{m=0}^{n-1}(2\lambda)^m\right)+4\left(\sum_{m=0}^{\infty}(2^m-1)\lambda^m-\sum_{m=0}^{n-1}(2^m-1)\lambda^m\right)\\
&\to 0 \mbox{ as }n\to \infty,
\end{aligned}
$$
which shows that $\{u_n\}$ is a Cauchy sequence on $(M,d)$. Then, from the completeness of $(M,d)$, we deduce that there exists $u^*\in M$ such that
$\lim_{n\to \infty}d(u_n,u^*)=0$. 	

The fact that $u^*$ is a fixed point of $F$ can be proved in a similar way as in Theorem~\ref{T3.4} only with the difference that instead of Proposition \ref{PR3.3} we use Proposition \ref{PR3.3H}.
\end{proof}

\vspace{0.5cm}

\noindent{\bf Author contributions:} All authors contributed equally to this work and have read and agreed to the published version of the manuscript.\\

\noindent{\bf Conflict of Interest:} This work does not have any conflicts of interest.\\


\end{document}